\documentclass[12pt]{article}

\usepackage[indentafter]{titlesec}
\titleformat{name=\section}{}{\thetitle.}{0.8em}{\centering\scshape}
\titleformat{name=\subsection}[runin]{}{\thetitle.}{0.5em}{\bfseries}[.]
\titleformat{name=\subsubsection}[runin]{}{\thetitle.}{0.5em}{\itshape}[.]
\titleformat{name=\paragraph,numberless}[runin]{}{}{0em}{}[.]
\titlespacing{\paragraph}{0em}{0em}{0.5em}
\titleformat{name=\subparagraph,numberless}[runin]{}{}{0em}{}[.]
\titlespacing{\subparagraph}{0em}{0em}{0.5em}

\usepackage{url}
\makeatletter
\newcommand{\addressa}[1]{\gdef\@addressa{#1}}
\newcommand{\addressb}[1]{\gdef\@addressb{#1}}
\newcommand{\emaila}[1]{\gdef\@emaila{\url{#1}}}
\newcommand{\emailb}[1]{\gdef\@emailb{\url{#1}}}
\newcommand{\sites}[1]{\gdef\@sites{\url{#1}}}
\newcommand{\@endstuff}{\par\vspace{\baselineskip}\noindent\small
\begin{tabular}{@{}l}
	\scshape \@addressa \\
	\textit{E-mail address:} \@emaila \\
	\textit{URL:} \@sites \\[2mm]
	\scshape \@addressb \\
	\textit{E-mail address:} \@emailb
\end{tabular}}
\AtEndDocument{\@endstuff}
\makeatother

\usepackage{titling}
\title{\normalsize\textbf{{\large F}INITELY {\large P}RESENTED {\large G}ROUPS WITH \\ {\large T}RANSCENDENTAL {\large S}PECTRAL {\large R}ADIUS}}
\author{Corentin Bodart and Denis Osin}
\date{\today}

\addressa{Mathematical Institute, University of Oxford, United Kingdom}
\emaila{corentin.bodart@maths.ox.ac.uk}
\sites{https://sites.google.com/view/corentin-bodart}

\addressb{Department of Mathematics, Vanderbilt University, Nashville 37240, U.S.A.}
\emailb{denis.v.osin@vanderbilt.edu}

\usepackage{import}

\usepackage[backend=biber, style=alphabetic, sorting=nyt,
    doi=false,isbn=false,url=false,eprint=false, maxnames=50]{biblatex}
\usepackage{csquotes}
\usepackage[english]{babel}
\usepackage{amsmath}
\usepackage{amsfonts}
\usepackage{amssymb}
\usepackage{mathrsfs} 
\usepackage{combelow} 

\usepackage{dsfont}
\usepackage{mathtools} 

\usepackage[dvipsnames]{xcolor}

\usepackage{hyperref}
\usepackage{faktor}
\usepackage{leftindex}
\usepackage{enumitem}
\usepackage[chapter]{easy-todo}         

\usepackage[left=3cm,right=3cm,top=2.8cm,bottom=3cm, marginparsep=3mm]{geometry}
\usepackage{changepage}
\usepackage{adjustbox}

\usepackage[retainorgcmds]{IEEEtrantools}   
\usepackage[font=small, hypcap=false]{caption}
\usepackage{parskip}

\usepackage{tikz-cd}
\usepackage{tikz}
\usetikzlibrary{shapes, shapes.geometric, arrows, arrows.meta, decorations.markings, decorations.pathreplacing, automata}
\pgfdeclarelayer{background} 
\pgfdeclarelayer{foreground}
\pgfsetlayers{background,main,foreground}
\DeclareMathOperator{\diam}{diam}
\DeclareMathOperator{\supp}{supp}

\newcommand{\CR}{\normalfont{(CR)} }

\newcommand{\Pc}{\mathcal P}

\newcommand{\say}[1]{``#1"}
\newcommand{\pare}[1]{{\normalfont (}#1\hspace*{1pt}{\normalfont )}}
\newcommand{\dif}{\mathrm{d}}

\newcommand{\N}{\mathbb N}
\newcommand{\Z}{\mathbb Z}
\newcommand{\Q}{\mathbb Q}
\newcommand{\R}{\mathbb R}
\newcommand{\C}{\mathbb C}

\newcommand{\PP}{\mathbb P} 

\newcommand{\Cay}{\mathcal C{ay}}

\newcommand{\onto}{\twoheadrightarrow}

\newcommand{\longto}{\longrightarrow}

\newcommand{\la}{\left\langle}
\newcommand{\ra}{\right\rangle}
\newcommand{\lla}{\left\langle\!\left\langle}
\newcommand{\rra}{\right\rangle\!\right\rangle}
\newcommand{\abs}[1]{\left| #1 \right|}
\newcommand{\norm}[1]{\left\|#1\right\|}

\renewcommand{\ge}{\geqslant}
\renewcommand{\le}{\leqslant}
\newcommand{\len}{\trianglelefteq}

\usepackage{xcolor}

\usepackage{amsthm}
\usepackage{chngcntr}

\theoremstyle{plain}
\newtheorem{thm}{Theorem}
\newtheorem*{thm*}{Theorem}
\newtheorem{lemma}[thm]{Lemma}
\newtheorem*{lemma*}{Lemma}
\newtheorem{cor}[thm]{Corollary}
\newtheorem*{cor*}{Corollary}
\newtheorem{prop}[thm]{Proposition}
\newtheorem*{prop*}{Proposition}

\theoremstyle{definition}
\newtheorem{defi}[thm]{Definition}
\newtheorem*{defi*}{Definition}
\newtheorem{rem}[thm]{Remark}
\newtheorem*{rem*}{Remark}
\newtheorem{question}[thm]{Question}
\newtheorem*{question*}{Question}

\begin{document}
\renewcommand*{\thethm}{\Alph{thm}}
\maketitle

\begin{abstract}
	\noindent We provide examples of groups with transcendental spectral radius: \smallskip
	
	\begin{itemize}[leftmargin=5mm]
		\item We first construct finitely presented examples, using links between decidability of the Word Problem and semi-computability of the spectral radius. This argument extends to the exponential growth rate and the asymptotic entropy.\smallskip
		
		\item We also construct a finitely generated example with decidable Word Problem, using classical small-cancellation theory. \smallskip
	\end{itemize}
	Along the way, we prove that $C'(1/6)$ groups satisfy the Rapid Decay property, and deduce some properties on their spectral radii of independent interest.
\end{abstract}

Throughout the paper, we consider \emph{marked groups}, that is, pairs $(G,S)$ with $G$ a group and $S$ a finite symmetric generating multiset (i.e., $S=S^{-1}$). We consider the simple random walk $(X_n)_{n\ge 0}$ on $\Cay(G,S)$, starting at $X_0=e_G$, and define
\[ p_{G,S}(g;n) = \PP[X_n=g].\]
Equivalently $p_{G,S}(g;n)=\frac1{\abs{S}^n}\, \#\{w\in S^n\mid \bar w=g\}$\footnote{Most of the time, $G$ and $S$ are clear from the context and $g=e_G$, hence we only write $p(n)$}, where $\bar w$ is the evaluation of the word $w$ as an element of $G$. The \emph{spectral radius} is defined as 
\[ \rho(G,S) = \limsup_{n\to\infty} \sqrt[n]{p_{G,S}(e;n)}. \]
This quantity depends on the global geometry of $\Cay(G,S)$, and is usually hard to compute. Computations of exact values are limited to virtually free groups, and some groups constructed from them (direct product, extension by an amenable group). In all these cases, the value $\rho(G,S)$ turns out to be algebraic (combining \cite{Muller_Schupp} and \cite{Kesten}).

This changed recently when Kassabov and Pak proved that some marked groups have transcendental spectral radius \cite{kassabov2024monotone}. However, their proof is non-constructive.

\medbreak

Our main result is the following:
\begin{thm}\label{thm:main}
	There exists a finitely presented group $G$ such that the spectral radius $\rho(G,S)$ is transcendental for all finite symmetric generating sets $S$.
\end{thm}
This answers Question 7.2 of \cite{kassabov2024monotone}. In Section \ref{sec:other_para}, Theorem \ref*{thm:main} is extended to the exponential growth rate and the asymptotic entropy.
\bigskip

Since the proof of Theorem \ref*{thm:main} uses crucially that $G$ has undecidable Word Problem, we propose another result of the same flavor:
\begin{thm}\label{thm:main2}
	There exists a finitely generated group $G$ with decidable Word Problem such that the spectral radius $\rho(G,S)$ is transcendental for a specific generating set $S$.
\end{thm}
The construction uses classical $C'(1/6)$ small-cancellation theory. As a key ingredient, we use that these groups satisfy the Rapid Decay property (uniformly):

\begin{thm}
	If $G=\la S\mid R\ra$ is a $C'(\frac16)$ presentation (with $R$ possibly infinite), then $(G,d_S)$ has the Rapid Decay property with $P(r)=(10r+1)^3$ \pare{see Definition \ref{def:RD_prop}}.
\end{thm}
Arzhantseva and Dru\cb{t}u have already proven that $C'(1/10)$ groups have the Rapid Decay property \cite{smallcancel_RD}. Our proof is a simplification of their, and gives a single polynomial $P(r)$. A version of this proof, due to the second author, was first circulated in 2014.

\medskip

As a byproduct, we prove the continuity of the spectral radius on subset of $C'(1/6)$ presentations in the space of marked groups:
\begin{cor}
	Let $G_i=\la S\mid R_i\ra$ be a sequence of $C'(1/6)$ presentations such that $(G_i,S)\to (G,S)$ in the space of marked groups, then $\rho(G_i,S) \to \rho(G,S)$.
\end{cor}
This means that, even though the spectral radius is a global invariant, it is determined by the local geometry of the Cayley graph, among $C'(1/6)$ groups. Of course, the same statement is not true in the entire space of marked groups, since free groups (with $\rho<1$) are residually finite and therefore limits of finite groups (with $\rho=1$).
\bigbreak

\textbf{Acknowledgments.} We are grateful to Anna Erschler, Tatiana Nagnibeda and Igor Pak for comments and encouragements, and to Carl-Fredrik Nyberg-Brodda for sparing a few generators and relators in the last construction of \S\ref{sec:main}. We thank the anonymous referees for their careful reading and useful explanations on entropy. The first author was supported by the Swiss NSF grant 200020-200400.

\newpage

\renewcommand*{\thethm}{\arabic{thm}}
\counterwithin{thm}{section}
\section{Semi-computable numbers}

The first ingredient in our proof is the following definition:
\begin{defi}
	A number $x\in\R$ is \emph{lower} (resp.\ \emph{upper}) \emph{semi-computable} if there exists an algorithm enumerating an increasing (resp.\ \emph{decreasing}) sequence $(x_k)$ of real algebraic numbers such that $x_k\to x$.
\end{defi}
(Real algebraic numbers $x_k$ can be specified as triplets $(P_k(X),a_k,b_k)\in\Q[X]\times\Q\times\Q$ such that $x_k$ is the only root of $P_k(X)$ in the interval $[a_k,b_k]$.)

\subsection{Lower bounds} We first observe that spectral radius are lower semi-computable in a very general setting, namely for all recursively presented groups:

\begin{lemma}\label{lem:rho_is_lower_semicomp}
	Let $G$ be a recursively presented group, then $\rho(G,S)$ is lower semi-computable. Moreover, the result is effective: there exists an algorithm with
	\begin{itemize}[leftmargin=18mm, rightmargin=2mm]
		\item[{\normalfont Input:}\hspace*{3mm}] A recursively presented marked group $(G,S)$, specified via a finite set $S$ and an algorithm enumerating a set of defining relations.
		\item[{\normalfont Output:}] An algorithm enumerating an increasing sequence of real algebraic numbers $(x_k)$ such that $x_k\to \rho(G,S)$.
	\end{itemize}
\end{lemma}
\begin{proof}
	Note that $p(m+n)\ge p(m)\cdot p(n)$, therefore Fekete's lemma ensures that
	\[ \rho(G,S) =  \sup_{n\ge 1} \sqrt[n]{p(n)}. \]
	As $G$ is recursively presented, its Word Problem (i.e., the set of words $w\in S^*$ such that $\bar w=e_G$) is recursively enumerable. Let $(v_i)_{i\ge 1}\subseteq S^*$ be a computable enumeration of the Word Problem, and let us define
	\[ p_k(n) = \frac1{\abs S^n} \cdot \#\{v_i : \ell(v_i)=n \text{ and } 1\le i\le k\}. \]
	Observe that $p_k(n)\nearrow p(n)$ pointwise  as $k\to\infty$. In particular
	\[ x_k = \max_{1\le n\le k}\sqrt[n]{p_k(n)}\]
	is a computable increasing sequence such that $x_k\to \rho(G,S)$, as promised.
\end{proof}

\begin{rem}
	The analogous result for \say{upper semi-computable} does not hold, even under the additional assumption that the Word Problem is decidable (with a Word Problem algorithm as part of the input). Indeed, this would imply that amenability is a co-semi-decidable property, contradicting \cite[Theorem 10.6]{Rauzy}.
	
	As will become clear in the proof of Theorem \ref*{thm:main}, there exist finitely presented groups (with undecidable Word Problem) whose spectral radius is not upper semi-computable.
\end{rem}

\section{Proof of Theorem \ref*{thm:main}} \label{sec:main}

\subsection{From semi-computability to the Word Problem} The main observation is
\begin{thm}\label{thm:comp_to_deci}
	Let $(G,S)$ be a recursively presented group such that the spectral radius $\rho(G,S)$ is upper semi-computable.
	Consider a subset $W\subseteq S^*$ such that, for each $w\in W$, either $\bar w=e_G$ or the normal subgroup $\lla \bar w\rra_G$ is non-amenable. Then there exists an algorithm which, given a word $w\in W$, decides whether $\bar w=e_G$ or not.
\end{thm}

(In particular, if $W=S^*$, then $G$ has decidable word problem.) This theorem relies crucially on a classical result due to Kesten.
\begin{thm}[{\cite{Kesten}}]
	Let $(G,S)$ be a marked group and $N\len G$ a non-amenable normal subgroup, then we have $\rho(G,S) < \rho(G/N,S)$.
\end{thm}

\begin{proof}[Proof of Theorem \ref*{thm:comp_to_deci}]
	Suppose that $G$ is presented as $\la S\mid R\ra$. Take $w\in W$, we run two parallel processes to decide whether $\bar w=e_G$ or not.
	\begin{itemize}[leftmargin=8mm]
		\item As $G$ is recursively presented, we can enumerate all words $v\in S^*$ such that $\bar v=e_G$. If at any point $v=w$, we stop the algorithm and conclude that $\bar w=e_G$.
		\item The group $G_w=G/\lla \bar w\rra_G=\la S\mid R,w\ra$ is recursively presented, hence $\rho(G_w,S)$ is lower semi-computable. We can compute two sequences $(x_k)$ and $(y_k)$ such that
		\[ x_k\searrow \rho(G,S) \quad\text{and}\quad  y_k\nearrow \rho(G_w,S). \]
		As $x_k$ and $y_k$ are real algebraic, we can compare them. If at any point $x_k<y_k$, we stop the algorithm and conclude that $\bar w\ne e_G$. It should be noted that, if $\bar w\ne e_G$, then the inequality $x_k<y_k$ will hold eventually as $\lla \bar w\rra_G$ is non-amenable and therefore $\rho(G,S)<\rho(G/\!\lla \bar w\rra_G,S)=\rho(G_w,S)$.
	\end{itemize}
	In both cases, the algorithm finishes and decides correctly if $\bar w=e_G$ or not.
\end{proof}

\bigbreak

\subsection{Finitely presented examples} 

We give two constructions using Theorem \ref{thm:comp_to_deci}.

\textbf{(a)} We recall a family of groups introduced by Higman \cite{Higman}: for each $I\subseteq \Z$,
	\[ H_I \coloneqq \la a,b,c,d \mid a^iba^{-i}=c^idc^{-i} \;\;(i\in I)\ra. \]
This is an amalgamated free product $A*_CB$ with $A=F(a,b)$, $B=F(c,d)$ and $C=F(\{g_i\}_{i\in I})$ (with injections $g_i\mapsto a^ib a^{-i}$ and $g_i\mapsto c^idc^{-i}$ respectively). We define
\[ w_i=a^iba^{-i}c^id^{-1}c^{-i}. \vspace*{-2mm}\]

\begin{lemma}
	If $i\notin I$, then $\lla \bar w_i\rra_{H_I}$ contains a non-abelian free group. \vspace*{-2mm}
\end{lemma}
\begin{proof}
	The subgroup $\la w_i,bw_ib^{-1}\ra$ is non-abelian free. This follows directly from the Normal Form Theorem for amalgamated free products, see eg.\ \cite[p. 186-187]{Lyndon_Schupp}.
\end{proof}

Consider $I\subset\Z$ which is recursively enumerable (hence $H_I$ is recursively presented), but not recursive. Using Higman's embedding theorem \cite{Higman}, we can embed $H_I$ into a finitely presented group $G$ (as Higman did). For any finite symmetric generating set $S$ of $G$, we consider a monoid homomorphism $\varphi\colon \{a^\pm,b^\pm,c^\pm,d^\pm\}^*\to S^*$ such that $\overline{\varphi(v)}=\bar v$ for all $v\in\{a^\pm,b^\pm,c^\pm,d^\pm\}$, and let $W=\{\varphi(w_i):i\in\Z\}\subset S^*$.
\begin{itemize}[leftmargin=8mm]
	\item Using the lemma, either $\bar w_i=e_G$ or $\lla \bar w_i\rra_G\ge \lla \bar w_i\rra_{H_I}$ is non-amenable.
	\item As $I$ is not recursive, there does not exist any algorithm which, given $\varphi(w_i)\in W$ (equivalently $i\in \Z$) decides whether $\bar w_i=e_G$ or not (equivalently $i\in I$ or not).
\end{itemize}
Therefore $\rho(G,S)$ cannot be upper semi-computable, hence cannot be algebraic. \vspace*{2mm} \hfill$\square$

\begin{rem}
	Using Karrass-Solitar theorem on amalgamated free products \cite{Karrass_Solitar}, one should be able to prove that $\lla g\rra_{H_I}$ contains a non-abelian free group for all $g\in H_I\setminus \{e\}$.
\end{rem}


\bigbreak

\textbf{(b)} Start with $H=\la S_0\mid R\ra$ a finitely presented group with undecidable Word Problem, first constructed in \cite{Boone,Novikov}. We consider $G=H*\la t\ra$ and $W=S_0^*$. For each $w\in W$, either $\bar w$ is trivial, or the normal subgroup it generates satisfy
\[ \lla \bar w\rra_G \ge \la \bar w,\, t\bar wt^{-1},\, t^2\bar wt^{-2}\ra \simeq C_n * C_n * C_n \ge F_2 \]
where $n\in\{2,3,\ldots\}\cup\{\infty\}$ is the order of $\bar w$. Again, Theorem \ref*{thm:comp_to_deci} implies that $\rho(G,S)$ cannot be upper semi-computable, hence is transcendental.

Note that, combining the main result of \cite{Borisov} with \cite[Theorem IV]{HNN}, we can find a group $H$ with undecidable Word Problem defined by $2$ generators and $12$ relators, and therefore $G$ defined by as little as $3$ generators and $12$ relators.

\section{Extension to other parameters} \label{sec:other_para}

Similarly to \cite{kassabov2024monotone}, the same ideas can be used to provide finitely presented groups for which the value of various parameters $f\colon\mathcal M\to\R$ (where $\mathcal M$ is the space of marked groups) are not computable. The key ingredients are the following:
\begin{itemize}[leftmargin=8mm]
	\item A recursively presented group $(G,S)$ and a subset $W\subseteq S^*$ such that
	\begin{itemize}[leftmargin=5mm]
		\item the Word Problem restricted to $W$ is undecidable, and
		\item $f(G,S)<f(G/\!\lla \bar w\rra_G,S)$ for all $w\in W$ such that $\bar w\ne e_G$.
	\end{itemize}
	\item For every recursively presented group $(Q,S)$, the value $f(Q,S)$ is lower semi-computable, and the result is effective in the sense of Lemma \ref{lem:rho_is_lower_semicomp}.
\end{itemize}
Under these conditions, we may conclude $f(G,S)$ is not upper semi-computable. In the following subsections, we check both conditions for the growth rate and the asymptotic entropy (up to a change of sign). In both cases, we may take $G = H*\Z$ where $H$ is any finitely presented group with undecidable Word Problem.

\subsection{Growth rates} Let $d_S$ be the word metric on $G$, and $\beta_{G,S}(n) = \abs{B_{G,S}(e,n)}$ the volume of the ball of radius $n$. The \emph{growth rate} of $(G,S)$ is defined as
\[ \omega(G,S) = \limsup_{n\to\infty}\sqrt[n]{\beta_{G,S}(n)}.\]
Anna Erschler constructed a continuum of (infinitely presented) marked groups with distinct growth rates \cite{Erschler_small_canc}, in particular containing transcendental values. Once again, our recipe provides finitely presented examples of the form $G=H*\Z$.
\begin{thm}[\cite{Sambusetti2002}]
	Let $H$ be a non-trivial group, then $G=H*\Z$ is \emph{growth tight}, that is, for every generating set $S$ and every $N\len G$ with $N\ne 1$, we have
	\[ \omega(G/N,S)<\omega(G,S).\]
\end{thm}
\begin{lemma} \label{lem:omega_is_upper_semicomp}
	Let $(G,S)$ be a recursively presented group, then the growth rate $\omega(G,S)$ is upper semi-computable. Moreover, the result is effective.
\end{lemma}
\begin{proof}
	The sequence $\beta_{G,S}(n)$ is sub-multiplicative, so Fekete's lemma implies that
	\[ \omega(G,S) = \inf_{n\ge 1}\sqrt[n]{\beta_{G,S}(n)}. \]
	Consider $\{(v_i,w_i)\}_{i\ge 1}$ a computable enumeration of
	\[ \Delta(G,S) = \{(v,w)\in S^*\times S^*\mid \bar v=\bar w\}. \]
	At time $k$, we define an equivalence relation $\sim_k$ on $S^*$, declaring $v_i\sim_k w_i$ for all $1\le i\le k$ (and then considering the reflexive, transitive closure). We define
	\[ P^{\le n}_k=\faktor{S^{\le n}}{\sim_k} \quad\text{and}\quad \beta_k(n) = \abs{P^{\le n}_k} \]
	Finally, let $y_k = \min\limits_{1\le n\le k} \sqrt[n]{\beta_k(n)}$. We have $\beta_k(n)\searrow \beta_{G,S}(n)$ hence $y_k\searrow \omega(G,S)$.
\end{proof}

\subsection{Asymptotic entropy} Let $(G,S)$ be a marked group. We consider the simple random walk $(X_n)_n$ on $\Cay(G,S)$. Recall that the Shannon entropy of $X_n$ is given by \vspace*{1mm}
\[ H(X_n) = -\sum_{g\in \supp X_n} p(g;n)\,\log\big(p(g;n)\big). \]
The \emph{asymptotic entropy} of $(G,S)$ is defined as 
\[ h(G,S)=\lim_{n\to\infty} \frac1n H(X_n). \]

Once again, both conditions are satisfied (for either construction of \S\ref{sec:main}):
\begin{thm} \label{thm:Kesten_for_entropy}
	Let $(G,S)$ be a group and $N\len G$ non-amenable, then
	\[ h(G/N,S)<h(G,S).\]
\end{thm}
This follows from results of Kaimanovich and Vershik. We recall some notations: for general probability measures $\mu$ on $G$, we denote by $\Gamma(G,\mu)$ and $h(G,\mu)$ the Poisson boundary and the asymptotic entropy of the $\mu$-random walk on $G$, respectively.
\begin{proof}
	Consider a measure $\mu$ of a finite entropy on $G$ (eg.\ the uniform measure on $S$), and let $\pi\colon G\onto G/N$ be the quotient map. \cite[Theorem 3.2]{KaimanovichVershik} states that
	\[ E_{(G,\mu)}\bigl(\Gamma(G/N,\pi_*\mu),\nu\bigr) \le h(G,\mu) \]
	with equality if and only if $\Gamma(G/N,\pi_*\mu)\simeq \Gamma(G,\mu)$, where $E_{(G,\mu)}$ denotes the differential entropy of a measure $G$-space and $\nu$ is the harmonic measure on $\Gamma(G/N,\pi_*\mu)$. Moreover,
	\begin{align*}
		E_{(G,\mu)}\bigl(\Gamma(G/N,\pi_*\mu),\nu\bigr)
		& := \sum_{g\in G}\mu(g) \int \log\left(\frac{\dif \nu}{\dif\, g^{-1}\nu}(\gamma)\right)\cdot\dif\nu(\gamma) \\
		& \ = \sum_{h\in G/N}\pi_*\mu(h) \int \log\left(\frac{\dif \nu}{\dif\, h^{-1}\nu}(\gamma)\right)\cdot\dif\nu(\gamma) \\
		& \ = E_{(G/N,\pi_*\mu)}\bigl(\Gamma(G/N,\pi_*\mu),\nu\bigr) = h(G/N,\pi_*\mu)
	\end{align*}
	using \cite[Theorem 3.1]{KaimanovichVershik} for this last equality. Finally, the main result of \cite{Kaimanovich_entropy} states that, if $N$ is non-amenable, then $\Gamma(G,\mu)$ and $\Gamma(G/N,\pi_*\mu)$ are not isomorphic. We conclude that $h(G/N,\pi_*\mu)=E\bigl(\Gamma(G/N,\pi_*\mu),\nu\bigr)< h(G,\mu)$.
\end{proof}
\bigbreak
\begin{lemma}\label{lem:h_is_upper_semicomp}
	Let $(G,S)$ be a recursively presented group, then the asymptotic entropy $h(G,S)$ is upper semi-computable. Moreover, the result is effective.
\end{lemma}
\begin{proof}
	As the sequence $H(X_n)$ is sub-additive, Fekete's lemma gives
	\[ h(G,S)= \inf_{n\ge 1} \frac1n H(X_n). \]
	Reusing the notations from Lemma \ref{lem:omega_is_upper_semicomp}, we define $P^n_k=\faktor{S^n}{\sim_k}$ and
	\[ H_k^n = -\sum_{A\in P_k^n} \frac{\abs A\,}{\,\abs S^n}\log\left(\frac{\abs A\,}{\,\abs S^n}\right).\]
	Observe that $H_k^n\searrow H(X_n)$ as $k\to\infty$. Finally, the sequence $x_k=\min_{1\le n\le k}\frac1n H_k^n$ provides upper approximations for $h(G,S)$.
\end{proof}

\bigbreak




\section{Property (RD) for $C'(1/6)$ groups}

\subsection{(RD) and (CR)}

Recall the definition of the Rapid Decay property of \cite{Jolissaint}:
\begin{defi}[Property (RD)] \label{def:RD_prop}
	Let $G$ be a group with a word metric $d_S$. We say that $G$ has the \emph{Rapid Decay property} if there exists a polynomial $P\in\R[X]$ such that
	\[ \norm{f}_{\mathrm{op}} \le P(r)\cdot \norm{f}_2 \]
	for all $f\in\C G$ supported on the ball $B_S(e,r)$.
\end{defi}
\begin{rem}
	The polynomial depends on choice of a word metric $d_S$, but the fact that $G$ satisfies the property does not.
\end{rem}

\medbreak

Throughout the years, many sufficient conditions for the (RD) property have been developed (see \cite{Centroid_property}). We will use a condition due to Chatterji and Ruane \cite{ChatterjiRuane}. We state it in the restricted setting of groups acting on themselves:
\begin{defi}[Property (CR)]A marked group $(G,d_S)$ has the \CR property if there exists a $G$-equivariant map $C\colon G\times G\to\Pc(G)$ (the powerset of $G$) such that
	\begin{enumerate}[leftmargin=8mm, label=(\alph*)]
		\item For all $x,y\in G$, we have $x\in C(x,y)$.
		\item For all $x,y,z\in G$, we have $C(x,y)\cap C(y,z)\cap C(z,x)\ne \emptyset$.
		\item There exists a polynomial $Q(r)$ such that, for all $x,y\in G$ and all $r\ge 0$,
		\[ \abs{C(x,y)\cap B_S(x,r)}\le Q(r). \vspace*{-2mm}\]
		\item There exists a polynomial $R(r)$ such that, for all $x,y\in G$ satisfying $d_S(x,y)\le r$, the diameter of $C(x,y)$ is bounded by $R(r)$.
	\end{enumerate}
\end{defi}
\begin{rem}
	Once again, the specific polynomials $Q(r)$ and $R(r)$ depends on the choice of the word metric $d_S$, but their mere existence does not.
\end{rem}

The link is made through the following result:
\begin{prop}[{\cite[Proposition 1.7]{ChatterjiRuane}}] \label{app:thm:CR}
	If $(G,S)$ has the \CR property for $Q(r)$ and $R(r)$, then it satisfies the Rapid Decay property for $P(r)=Q(R(r))^\frac32$.
\end{prop}

\bigbreak

\subsection{Strebel's description of geodesics in $C'(1/6)$ groups}

Let us recall some geometric properties of Cayley graphs of $C'(1/6)$ presentations:

\medbreak

\begin{lemma}[{\cite[Lemma 41]{strebel}}] \label{app:lem}
	Consider a path labeled by $u$ with $uv\in R$.
	\begin{itemize}[leftmargin=8mm]
		\item If $\abs{u}<\abs{v}$, then $u$ is the unique geodesic between its endpoints.
		\item If $\abs{u}=\abs{v}$, then $u$ and $v$ are the only two geodesics between their endpoints.
	\end{itemize}
\end{lemma} 

\vspace*{4mm}

\begin{cor}
	Let $\alpha$ be a geodesic segment in $\Cay(G,S)$, and let $\beta$ be a cycle labeled by a relation $r\in R$. Then $\alpha\cap \beta$ is either empty or connected. \vspace*{-2mm}
\end{cor}
\begin{proof}
	Consider $p,q\in \alpha\cap\beta$. Using Lemma \ref*{app:lem}, any geodesic between $p$ and $q$ has to be one of the two arcs of $\beta$ delimited by $p$ and $q$. In particular, this holds for the segment of $\alpha$ between $p$ and $q$.
\end{proof}

\begin{thm}[{\cite[Theorem 43]{strebel}}] \label{app:thm:str}
	If $T$ is a simple\footnote{meaning the sides are disjoint} geodesic triangle, then any reduced van Kampen diagram it spans is of one of the following forms: 
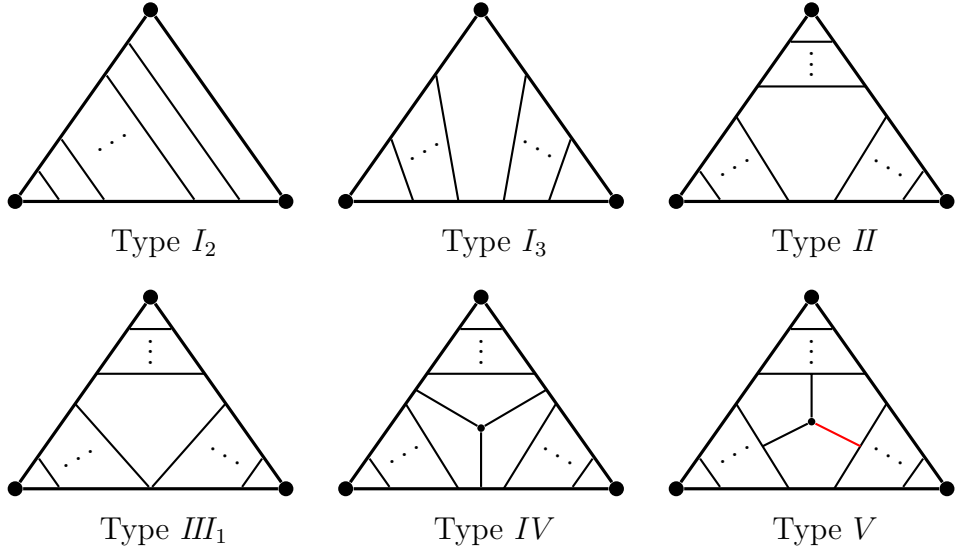
\vspace*{2mm}
	\begin{center}
		\begin{tikzpicture}[xscale=.6, yscale=.85]
			\node[circle, fill=black, inner sep=2pt] (x) at (0,0) {};
			\node[circle, fill=black, inner sep=2pt] (y) at (6,0) {};
			\node[circle, fill=black, inner sep=2pt] (z) at (3,3) {};
			
			\node[inner sep=0] (xy1) at (1,0) {};
			\node[inner sep=0] (xy2) at (2,0) {};
			\node[inner sep=0] (xy3) at (3,0) {};
			\node[inner sep=0] (xy4) at (4,0) {};
			\node[inner sep=0] (xy5) at (5,0) {};
			\node[inner sep=0] (xz1) at (.5,.5) {};
			\node[inner sep=0] (xz2) at (1,1) {};
			\node[inner sep=0] (xz3) at (1.5,1.5) {};
			\node[inner sep=0] (xz4) at (2,2) {};
			\node[inner sep=0] (xz5) at (2.5,2.5) {};
			
			\draw[very thick] (x) -- (y) -- (z) -- (x);
			\draw[thick] 	(xy1) -- (xz1)
			(xy2) -- (xz2)
			(xy4) -- (xz4)
			(xy5) -- (xz5);
			{\normalfont
				\node at (3.3,-.66) {Type $I_2$};
				\node[rotate=30] at (2.2,.95) {$\dots$};
			}
		\end{tikzpicture} \hspace*{3mm}
		\begin{tikzpicture}[xscale=.6, yscale=.85]
			\node[circle, fill=black, inner sep=2pt] (x) at (0,0) {};
			\node[circle, fill=black, inner sep=2pt] (y) at (6,0) {};
			\node[circle, fill=black, inner sep=2pt] (z) at (3,3) {};
			
			\node[inner sep=0] (xy1) at (1,0) {};
			\node[inner sep=0] (xy2) at (2,0) {};
			\node[inner sep=0] (xy3) at (3,0) {};
			\node[inner sep=0] (xy4) at (4,0) {};
			\node[inner sep=0] (xy5) at (5,0) {};
			\node[inner sep=0] (xz1) at (.5,.5) {};
			\node[inner sep=0] (xz2) at (1,1) {};
			\node[inner sep=0] (xz3) at (1.5,1.5) {};
			\node[inner sep=0] (xz4) at (2,2) {};
			\node[inner sep=0] (xz5) at (2.5,2.5) {};
			\node[inner sep=0] (yz1) at (5.5,.5) {};
			\node[inner sep=0] (yz2) at (5,1) {};
			\node[inner sep=0] (yz3) at (4.5,1.5) {};
			\node[inner sep=0] (yz4) at (4,2) {};
			\node[inner sep=0] (yz5) at (3.5,2.5) {};
			
			\draw[very thick] (x) -- (y) -- (z) -- (x);
			\draw[thick] 	(1.5,0) -- (xz2)
			(2.5,0) -- (xz4)
			(3.5,0) -- (yz4)
			(4.5,0) -- (yz2);
			{\normalfont
				\node at (3.3,-.66) {Type $I_3$};
				\node[rotate=30] at (1.8,.8) {$\dots$};
				\node[rotate=-30] at (4.3,.8) {$\dots$};
			}
		\end{tikzpicture} \hspace*{3mm}
		\begin{tikzpicture}[xscale=.6, yscale=.85]
			\node[circle, fill=black, inner sep=2pt] (x) at (0,0) {};
			\node[circle, fill=black, inner sep=2pt] (y) at (6,0) {};
			\node[circle, fill=black, inner sep=2pt] (z) at (3,3) {};
			
			\node[inner sep=0] (xy1) at (1,0) {};
			\node[inner sep=0] (xy2) at (2,0) {};
			\node[inner sep=0] (xy3) at (3,0) {};
			\node[inner sep=0] (xy4) at (4,0) {};
			\node[inner sep=0] (xy5) at (5,0) {};
			\node[inner sep=0] (xz1) at (.5,.5) {};
			\node[inner sep=0] (xz2) at (1,1) {};
			\node[inner sep=0] (xz3) at (1.5,1.5) {};
			\node[inner sep=0] (xz4) at (2,2) {};
			\node[inner sep=0] (xz5) at (2.5,2.5) {};
			\node[inner sep=0] (yz1) at (5.5,.5) {};
			\node[inner sep=0] (yz2) at (5,1) {};
			\node[inner sep=0] (yz3) at (4.5,1.5) {};
			\node[inner sep=0] (yz4) at (4,2) {};
			\node[inner sep=0] (yz5) at (3.5,2.5) {};
			
			\draw[very thick] (x) -- (y) -- (z) -- (x);
			\draw[thick] 	(xy1) -- (xz1)
			(2.5,0) -- (1.33,1.33)
			(xy5) -- (yz1)
			(3.5,0) -- (4.66,1.33)
			(1.8,1.8) -- (4.2,1.8)
			(xz5) -- (yz5);
			{\normalfont
				\node at (3.3,-.66) {Type $I\!I$};
				\node[rotate=30] at (1.35,.55) {$\dots$};
				\node[rotate=-30] at (4.75,.5) {$\dots$};
				\node at (3,2.27) {$\vdots$};
			}
		\end{tikzpicture}
		
		\begin{tikzpicture}[xscale=.6, yscale=.85]
			\node[circle, fill=black, inner sep=2pt] (x) at (0,0) {};
			\node[circle, fill=black, inner sep=2pt] (y) at (6,0) {};
			\node[circle, fill=black, inner sep=2pt] (z) at (3,3) {};
			
			\node[inner sep=0] (xy1) at (1,0) {};
			\node[inner sep=0] (xy3) at (3,0) {};
			\node[inner sep=0] (xy5) at (5,0) {};
			\node[inner sep=0] (xz5) at (2.5,2.5) {};
			\node[inner sep=0] (yz1) at (5.5,.5) {};
			\node[inner sep=0] (yz5) at (3.5,2.5) {};
			
			\draw[very thick] (x) -- (y) -- (z) -- (x);
			\draw[thick] 	(xy1) -- (xz1)
			(xy3) -- (1.33,1.33)
			(xy5) -- (yz1)
			(xy3) -- (4.66,1.33)
			(1.8,1.8) -- (4.2,1.8)
			(xz5) -- (yz5);
			{\normalfont
				\node at (3.3,-.66) {Type $I\!I\!I_1$};
				\node[rotate=30] at (1.45,.5) {$\dots$};
				\node[rotate=-30] at (4.7,.5) {$\dots$};
				\node at (3,2.27) {$\vdots$};
			}
		\end{tikzpicture}	\hspace*{3mm}
		\begin{tikzpicture}[xscale=.6, yscale=.85]
			\node[circle, fill=black, inner sep=2pt] (x) at (0,0) {};
			\node[circle, fill=black, inner sep=2pt] (y) at (6,0) {};
			\node[circle, fill=black, inner sep=2pt] (z) at (3,3) {};
			
			\node[inner sep=0] (xy1) at (1,0) {};
			\node[inner sep=0] (xy5) at (5,0) {};
			\node[inner sep=0] (xz1) at (.5,.5) {};
			\node[inner sep=0] (xz5) at (2.5,2.5) {};
			\node[inner sep=0] (yz1) at (5.5,.5) {};
			\node[inner sep=0] (yz2) at (5,1) {};
			\node[inner sep=0] (yz3) at (4.5,1.5) {};
			\node[inner sep=0] (yz4) at (4,2) {};
			\node[inner sep=0] (yz5) at (3.5,2.5) {};
			\node[inner sep=1pt, circle, fill=black] (m) at (3,.95) {};
			
			\draw[very thick] (x) -- (y) -- (z) -- (x);
			\draw[thick] 	(xy1) -- (xz1)
			(2.5,0) -- (1.33,1.33)
			(xy5) -- (yz1)
			(3.5,0) -- (4.66,1.33)
			(1.8,1.8) -- (4.2,1.8)
			(xz5) -- (yz5)
			(m) -- (xy3)
			(m) -- (1.55,1.55)
			(m) -- (4.45,1.55);
			{\normalfont
				\node at (3.3,-.66) {Type $IV$};
				\node[rotate=30] at (1.35,.55) {$\dots$};
				\node[rotate=-30] at (4.75,.5) {$\dots$};
				\node at (3,2.27) {$\vdots$};
			}
		\end{tikzpicture} \hspace*{3mm}
		\begin{tikzpicture}[xscale=.6, yscale=.85]
			\node[circle, fill=black, inner sep=2pt] (x) at (0,0) {};
			\node[circle, fill=black, inner sep=2pt] (y) at (6,0) {};
			\node[circle, fill=black, inner sep=2pt] (z) at (3,3) {};
			
			\node[inner sep=0] (xy1) at (1,0) {};
			\node[inner sep=0] (xy3) at (3,0) {};
			\node[inner sep=0] (xy5) at (5,0) {};
			\node[inner sep=0] (xz1) at (.5,.5) {};
			\node[inner sep=0] (xz5) at (2.5,2.5) {};
			\node[inner sep=0] (yz1) at (5.5,.5) {};
			\node[inner sep=0] (yz5) at (3.5,2.5) {};
			\node[inner sep=1pt, circle, fill=black] (m) at (3,1.05) {};
			
			\draw[very thick] (x) -- (y) -- (z) -- (x);
			\draw[thick] 	(xy1) -- (xz1)
			(2.5,0) -- (1.33,1.33)
			(xy5) -- (yz1)
			(3.5,0) -- (4.66,1.33)
			(1.8,1.8) -- (4.2,1.8)
			(xz5) -- (yz5)
			(m) -- (3,1.8)
			(m) -- (1.916,0.666);
			\draw[thick, red] (m) -- (4.083,0.666);
			{\normalfont
				\node at (3.3,-.66) {Type $V$};
				\node[rotate=30] at (1.35,.55) {$\dots$};
				\node[rotate=-30] at (4.75,.5) {$\dots$};
				\node at (3,2.27) {$\vdots$};
			}
		\end{tikzpicture}
	\end{center}
\end{thm}
A recurring observation is that every interior (resp.\ exterior) consolidated\footnote{i.e.\ concatenation of edges seen as a single segment on the previous diagrams, eg.\ the red segment} edge of a face labeled by $r\in R$ has length $<\frac16\abs{r}$ (resp.\ $\le\frac12\abs r$).

\bigbreak

\subsection{Proof of property (CR)}

For each $x,y\in G$, we pick $[x,y]$ a geodesic between $x$ and $y$. (In a $G$-equivariant way: pick the geodesic with ShortLex labelling.)

\medbreak

We define $C(x,y)$ as the union of $[x,y]$ with all loops $\gamma$ labeled by relations such that $\abs{\gamma\cap[x,y]}\ge \frac16\abs{\gamma}$ (meaning the number of common edges). We consider $C(x,y)$ as a subgraph, formed by all vertices and edges of this geodesic and these cycles. \vspace*{1mm}

We now check that this construction satisfies Property (CR).
\begin{enumerate}[leftmargin=8mm, label=(\alph*)]
	\item This is obvious.
	\item Fix $x,y,z\in G$. Let $p_x$ be the point on $[x,y]\cap [z,x]$ furthest from $x$, and similarly for $p_y,p_z$. It follows that $[p_x,p_y],[p_y,p_z],[p_z,p_x]$ form a simple geodesic triangle.
	
	In each of the $6$ types above, any cycle $\gamma$ around a face has a most $6$ sides, hence one of its exterior edges $\gamma\cap [p_a,p_b]$ has length $\ge\frac16\abs{\gamma}$ (for $a,b\in\{x,y,z\}$). It follows that $\gamma\subseteq C(a,b)$. An easy case-by-case analysis now implies that
	\[ C(x,y)\cap C(y,z)\cap C(z,x)\ne\emptyset.\]
	
	\item \textbf{Claim.} For every $p\in C(x,y)$, we have $d_S(x,p)\ge d_{C(x,y)}([x,y],p)$. 
	\begin{proof}[Proof of the claim]
		If $p\in[x,y]$, the statement is obvious. Suppose $p\in \gamma$. We consider $[x,p]$ a geodesic between $x$ and $p$. Let $p'$ be the point on $[x,p]\cap\gamma$ furthest from $p$, and $x'$ be the point on $[x,p]\cap[x,y]$ furthest from $x$ (see Figure \ref*{app:fig:wlog}). Observe that $d_S(x',p')\ge d_{C(x',y)}([x',y],p')$ implies the desired inequality, hence we may suppose w.l.o.g.\ $p=p'$ and $x=x'$. 
		\begin{center}
			\begin{minipage}{.45\linewidth}
				\centering
				\begin{tikzpicture}[scale=.8]
					\node[circle, fill=black, inner sep=1.5pt, label=below:$x$] (x) at (-3.5,0) {};
					\node[circle, fill=black, inner sep=1.5pt, label={[label distance=-3.5pt]below:$x'$}] (x') at (-2.5,0) {};
					\node[circle, fill=black, inner sep=1.5pt, label=below:$y$] (y) at (2.5,0) {};
					\node[circle, fill=black, inner sep=1.5pt, label=above:$p'$] (p') at (-.3,1.5) {};
					\node[circle, fill=black, inner sep=1.5pt, label=above:$p$] (p) at (1,1.5) {};
					
					\draw[very thick] (x) -- (x') -- (y);
					\draw[very thick] (x') to [out=25, in=-160, looseness=.9] (p') to [out=20, in=160] (p);
					\draw[very thick, dashed, Magenta] (-.7,.8)
					to [out=90, in=-160, looseness=.7] (p')
					to [out=20, in=160] (p)
					to [out=-20, in=100, looseness=.7] (1.5,.8);
					\node[Magenta] at (1.7,1.2) {$\gamma$};
				\end{tikzpicture} \vspace*{-2.5mm}
				\captionof{figure}{$p$, $p'$, $x$ and $x'$.} \label{app:fig:wlog}
			\end{minipage}
			\begin{minipage}{.45\linewidth}
				\centering
				\begin{tikzpicture}[scale=.8]
					\fill[Turquoise!15] (0,1.5) -- (-3,0) -- (1,0)
					to [out=105, in=-45] (.5,.53)
					to [out=135, in=-85] (0,1.5);
					\fill[Turquoise!30] (.3,0) -- (1,0) to [out=105, in=-45] (.5,.53) -- (.3,0);
					\draw[thick, black] (.3,0) -- (.5,.53);
					\node at (.62,.18) {\footnotesize$r$};
					
					\draw[Magenta, very thick] (1,0)
					to [out=105, in=-45] (.5,.53)
					to [out=135, in=-85] (0,1.5)
					to [out=95, in=-120] (.25,2.2);
					\node[Magenta] at (.5,1) {$\gamma$};
					
					\node[circle, fill=black, inner sep=1.5pt, label=below:$x$] (x) at (-3,0) {};
					\node[circle, fill=black, inner sep=1.5pt, label=above:$p$] (p) at (0,1.5) {};
					\node[circle, fill=black, inner sep=1.5pt, label=below:$q$] (q) at (1,0) {};
					\draw[very thick] (p) -- (x) -- (q) -- (1.5,0);
				\end{tikzpicture} \vspace*{-2mm}
				\captionof{figure}{A relation in the $q$ corner.} \label{app:fig:corner}
			\end{minipage}
		\end{center}
		We pick $q$ an endpoint of $\gamma\cap[x,y]$ such that $\abs{[p,q]}\le\frac12\abs\gamma$ where $[p,q]$ is the corresponding arc of $\gamma$. If both endpoints satisfy this condition, we pick $q$ closer to $x$. We consider the simple geodesic triangle formed by $[x,q]$, $[x,p]$, $[p,q]$. 
		
		Observe that the corner $q$ cannot be filled by a triangle relation (Figure \ref*{app:fig:corner}). Indeed $\abs{r\cap[x,p]}\le\frac12\abs r$, $\abs{r\cap [p,q]}<\frac16\abs{r}$ and the interior edges (at most two of them in Type $V$) have length $<\frac16\abs r$. The same is true for the corner $p$. It follows that the diagram spanned has to be of type $I_2$ (see Figure \ref*{app:fig:I2}).
		\begin{center}
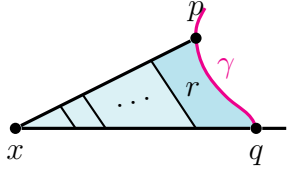

			\begin{tikzpicture}[scale=.8]
				\fill[Turquoise!15] (0,1.5) -- (-3,0) -- (1,0)
				to [out=105, in=-45] (.5,.53)
				to [out=135, in=-85] (0,1.5);
				\fill[Turquoise!30] (0,1.5) -- (-.75,1.133) -- (0,0) -- (1,0)
				to [out=105, in=-45] (.5,.53)
				to [out=135, in=-85] (0,1.5);
				\node at (-.05,.65) {$r$};
				\node[rotate=15] at (-1,.45) {$\dots$};
				\draw[Magenta, very thick] (1,0)
				to [out=105, in=-45] (.5,.53)
				to [out=135, in=-85] (0,1.5)
				to [out=95, in=-120] (.15,2);
				\node[Magenta] at (.5,1) {$\gamma$};
				
				\node[circle, fill=black, inner sep=1.5pt, label=below:$x$] (x) at (-3,0) {};
				\node[circle, fill=black, inner sep=1.5pt, label=above:$p$] (p) at (0,1.5) {};
				\node[circle, fill=black, inner sep=1.5pt, label=below:$q$] (q) at (1,0) {};
				\draw[very thick] (p) -- (x) -- (q) -- (1.5,0);
				\draw[thick] (-2,0) -- (-2.25,.366);
				\draw[thick] (-1.5,0) -- (-1.875,.55);
				\draw[thick] (0,0) -- (-.75,1.133);
			\end{tikzpicture} \vspace*{-2mm}
			\captionof{figure}{The true configuration.} \label{app:fig:I2}
		\end{center}
		Note that, if $r=\gamma$, then $\abs{r\cap [x,p]}=0$ and the interior edge has length $<\frac16\abs r$, and the second arc from $p$ to $[x,y]$ has length $\le \frac12\abs\gamma$ and is closer to $x$, which contradicts the definition of $q$. We conclude that $r\ne\gamma$, hence $\abs{r\cap[x,q]}\le\frac12\abs r$, $\abs{[p,q]}=\abs{r\cap[p,q]}<\frac16\abs r$ and the interior edge is $<\frac16\abs r$. This forces
		\[ d_S(x,p) \ge \abs{[x,p]\cap r} \ge \frac16\abs r. \]
		In particular, we get $d_{C(x,y)}(x,y)\le \abs{[p,q]}<\frac16\abs r\le d_S(x,p)$.
	\end{proof}
	
	\medbreak
	
	Hence, for every $p\in B_S(e,r)\cap C(x,y)$, there exists a relation $\gamma$ and an endpoint $q$ of $\gamma\cap[x,y]$ such that $d_{C(x,y)}(p,q)\le r$. In particular, $d_S(x,q)\le 2r$. Note that each point $q\in[x,y]$ is the endpoint of at most two segments $\gamma\cap[x,y]$. Indeed we cannot have $\gamma_1\cap[x,y]\subseteq \gamma_2\cap[x,y]$ by the $C'(1/6)$ condition. It follows that
	\[ \abs{B_S(e,r)\cap C(x,y)} \le (2r+1)^2 =: Q(r). \]
	
	\item Fix $x,y\in G$. Let $p,q\in C(x,y)$. We first suppose that $p\in \gamma_1$ and $q\in \gamma_2$, with $\gamma_1\cap\gamma_2\cap[x,y]\ne\emptyset$. Then  \vspace*{-6mm}
	\begin{center}
		\begin{minipage}{.45\linewidth}
			\begin{align*}
				d_S(p,q)
				& \le \min\{\ell_1+a+\ell_2;\; r_1+c+r_2\} \\
				& \le \frac12(\ell_1+r_1+\ell_2+r_2+a+c) \\
				& \le \frac12\big(5(a+b)+5(b+c)+a+c\big) \\
				& \le 5(a+b+c) \le 5 \, d_S(x,y).
			\end{align*}
		\end{minipage}\hspace*{2mm}
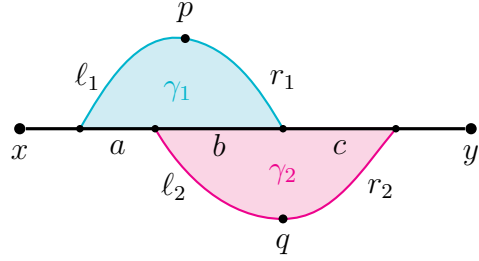
		\begin{minipage}{.45\linewidth}
			\centering
			\begin{tikzpicture}[baseline=(b.base)]
				\draw[thick, Turquoise, fill=Turquoise!20] (-2.2,0)
				to [out=60, in=170, looseness=.9] (-.8,1.2)
				to [out=-10, in=120, looseness=.9] (.5,0) to (-2.2,0);
				\draw[thick, Magenta, fill=Magenta!20] (-1.2,0)
				to [out=-60, in=180, looseness=.9] (.5,-1.2)
				to [out=0, in=-130, looseness=.9] (2,0) to (-1.2,0);
				\node[Turquoise] at (-.9,.5) {$\gamma_1$};
				\node[Magenta] at (.5,-.6) {$\gamma_2$};
				
				\node[circle, fill=black, inner sep=1.5pt, label=below:$x$] (x) at (-3,0) {};
				\node[circle, fill=black, inner sep=1.5pt, label=below:$y$] (y) at (3,0) {};
				\draw[very thick] (x) -- (y);
				
				\node[circle, fill=black, inner sep=1pt] (l1) at (-2.2,0) {};
				\node[circle, fill=black, inner sep=1pt] (l2) at (-1.2,0) {};
				\node[circle, fill=black, inner sep=1pt] (r1) at (.5,0) {};
				\node[circle, fill=black, inner sep=1pt] (r2) at (2,0) {};
				\node[circle, fill=black, inner sep=1.2pt, label=above:$p$] (p) at (-.8,1.2) {};
				\node[circle, fill=black, inner sep=1.2pt, label=below:$q$] (q) at (.5,-1.2) {};				
				
				\node at (-1.7,-.25) {$a$};
				\node at (-.35,-.22) {$b$};
				\node at (1.25,-.25) {$c$};
				\node at (-2.1,.7) {$\ell_1$};
				\node at (.5,.65) {$r_1$};
				\node at (-.95,-.77) {$\ell_2$};
				\node at (1.8,-.8) {$r_2$};
				
				\node (b) at (0,2) {};
			\end{tikzpicture}
		\end{minipage}
	\end{center}
	In the other cases ($p$ or $q$ on $[x,y]$; $\gamma_1,\gamma_2$ disjoints), we have $d_S(p,q)\le 3d_S(x,y)$. In conclusion, we have $\diam_S(C(x,y)) \le 5 \, d_S(x,y)$, so $R(r)=5r$ satisfies.
\end{enumerate}
This finishes the proof of the (CR) property. Combining this with Proposition \ref{app:thm:CR}, we deduce the Rapid Decay property in a uniform and effective way:

\begin{thm}
	Let $G=\la S\mid R\ra$ be a $C'(1/6)$ presentation {\normalfont(}\hspace*{-.5mm}with $R$ possibly infinite{\normalfont)}, then $(G,S)$ has the Rapid Decay property with $P(r)=(10r+1)^3$.
\end{thm}

\section{Proof of Theorem \ref*{thm:main2}} \label{sec:mainB}

\subsection{Upper semi-computability} We find upper approximations for $\rho(G,S)$. This requires control on the speed of convergence of $\sqrt[2n]{p(2n)}\to\rho(G,S)$ (and some way to compute $p(2n)$). This control is given to us by the Rapid Decay property, via the following observation (taking $f(x)=p(x;n)$ in Definition \ref{def:RD_prop}):
\begin{thm}[{Chatterji, Pittet, Saloff-Coste, see  \cite[Theorem 1.3]{RD_random_walks}}] If $(G,d_S)$ has the Rapid Decay property for a given polynomial $P$, then
	\[ \rho(G,S)^{2n} \le P(n)^2 \cdot p(2n).\]
\end{thm}

\begin{cor} \label{cor:rho_cont}
	Let $G_i=\la S\mid R_i\ra$ be a sequence of $C'(1/6)$ presentations such that $(G_i,S)\to (G,S)$ in the space of marked groups, then $\rho(G_i,S) \to \rho(G,S)$.
\end{cor}
\begin{proof}
	Combining the previous results, we have
	\begin{equation} \tag{$*$}
		\rho(G_i,S)/\sqrt[n]{P(n)}\le \sqrt[2n]{p_{G_i,S}(2n)} \le \rho(G_i,S),
	\end{equation}
	in particular $\sqrt[2n]{p_{G_i,S}(2n)} \to \rho(G_i,S)$ uniformly. It follows that
	\begin{align*}
		\lim_{i\to\infty}\rho(G_i,S) & = \lim_{i\to\infty}\lim_{n\to\infty}\sqrt[2n]{p_{G_i,S}(2n)} \\
		& \overset!= \lim_{n\to\infty}\lim_{i\to\infty}\sqrt[2n]{p_{G_i,S}(2n)} = \lim_{n\to\infty}\sqrt[2n]{p_{G,S}(2n)}=\rho(G,S). \qedhere
	\end{align*}
\end{proof}
\begin{cor} \label{cor:rho_is_comp}
	If $G=\la S\mid R\ra$ is a recursive $C'(\frac16)$ presentation, then $\rho(G,S)$ is upper semi-computable {\normalfont(}hence computable{\normalfont)}. Moreover, the result is effective.
\end{cor}
\begin{proof} 
	Rewriting the equation $(*)$, we get that
	\[ \sqrt[2n]{p(2n)} \le \rho(G,S) \le \sqrt[n]{P(n)} \sqrt[2n]{p(2n)}. \]
	Moreover, the Word Problem is decidable (using Dehn's algorithm), hence $p(2n)$ is computable. This provides both upper and lower bounds.
\end{proof}

\subsection{Construction}

More precisely, we prove the following result:

\begin{thm} \label{thm:diagonal}
Let $X=\{x_n\}_{n\ge 1}$ be recursively enumerable set of computable reals, i.e., such that there exists an algorithm with specifications
\begin{itemize}[leftmargin=18mm, rightmargin=2mm]
	\item[{\normalfont Input:}\hspace*{3mm}] An index $n$.
	\item[{\normalfont Output:}] Two algorithms providing two monotonous sequences of rational numbers $(a_{m,n})_{m\ge 1}$ and $(b_{m,n})_{m\ge 1}$ such that $a_{m,n}\nearrow x_n$ and $b_{m,n}\searrow x_n$ as $m\to\infty$.
\end{itemize}
Then there exists a marked group $(G,S)$ with decidable Word Problem s.t.\ $\rho(G,S)\notin X$.
\end{thm}
We combines small-cancellation theory and a diagonal argument. This combination of ideas already appears in \cite{Erschler_small_canc}. The main additional difficulty is computability.

\begin{proof}
	Fix an alphabet $S$, and consider an infinite sequence of cyclically reduced words $(r_i)_{i\in\N}\subset S^*$ satisfying the $C'(1/6)$ condition. For instance, we can take $S=\{a,b\}$ and $r_i=(a^ib^i)^7$ as in \cite{Erschler_small_canc}.
	For $I\subseteq\N$, we define $H_I = \la S\mid \{r_i\}_{i\in I}\ra$.
	
	We define two sequences $(i_k)\subset\N$ and $(\varepsilon_i)\subset\Q_{>0}$, with $(i_k)$ increasing, such that
	\[ \forall 1\le j\le k, \quad \abs{\rho(H_{I_k},S) - x_j}>\varepsilon_j, \]
	where $I_k = \{i_1,\ldots,i_k\}$. Suppose that $(i_j)_{j\le k}$ and $(\varepsilon_j)_{j\le k}$ are fixed. Observe that
	\begin{enumerate}[leftmargin=8mm, label=(\arabic*)]
		\item $H_I=\la S\mid \{r_i\}_{i\in I}\ra$ is acylindrically hyperbolic \cite[Theorem 1.3]{smallcancel_acyli}, hence every infinite normal subgroup is acylindrically hyperbolic \cite[Corollary 1.5]{acyli}. This covers $\lla r_\ell\rra_{H_I}$ for $\ell\notin I$. ($r_\ell^n$ does not contain half of a relation $r_i$, hence $\la r_\ell\ra$ is infinite by Greendlinger's lemma.) It follows that $\lla r_\ell\rra_{H_I}$ is not amenable hence
		\[ \forall \ell\notin I,\quad \rho(H_I,S) < \rho(H_{I\cup\{\ell\}},S). \]
		
		\item Greendlinger's lemma implies that the Cayley graphs $\Cay(H_I,S)$ and $\Cay(H_J,S)$ coincide up to radius $\frac12\min_{i\in I\triangle J}\abs{r_i}$. In particular, we have $(H_{I\cup\{\ell\}},S)\to (H_I,S)$ in the space of marked groups. Corollary \ref{cor:rho_cont} now implies that
		\[ \lim_{\ell\to\infty}\rho(H_{I_k\cup\{\ell\}},S) = \rho(H_{I_k},S). \]
	\end{enumerate}
	This implies that, for $\ell$ large enough, we have
	\[ \forall 1\le j\le k, \quad \abs{\rho(H_{I_k\cup\{\ell\}},S) - x_j}>\varepsilon_j \quad\text{and}\quad \rho(H_{I_k\cup\{\ell\}},S)\ne x_{i+1}. \] 
	Therefore, we can estimate in parallel (dovetailing) all values $x_j$ for $1\le j\le k+1$ and $\rho(H_{I_k\cup\{\ell\}},S)$ for $\ell>i_k$ (Corollary \ref{cor:rho_is_comp}), until we find a certificate that some $\ell$ satisfies the previous conditions. We take $i_{k+1}=\ell$ and $0<\varepsilon_{k+1}<\abs{\rho(H_{I_k\cup\{\ell\}},S)-x_{k+1}}$.
	
	\medbreak
	
	Finally, we consider $G=H_{\{i_1,i_2,\ldots\}}$. By construction, we have $\abs{\rho(G,S)-x_j}\ge \varepsilon_j$ for all $j$, hence $\rho(G,S)\notin X$. Moreover, $G$ is given by a recursive $C'(1/6)$ presentation and therefore has decidable Word Problem (using Dehn's algorithm).
\end{proof}

\begin{rem}
	Observations (1) and (2) prove that the family $(H_I)_{I\subseteq\N}$ satisfies the conclusions of \cite[Lemma 1.11]{kassabov2024monotone}. In particular, this provides a family of $2$-generated groups for which the spectral radius $\rho(H_I,S)$ (or any other strictly monotone parameter) takes a continuum of values. Another family satisfying these conditions is Higman's family $(H_{\bar I})_{I\subseteq \Z}$ used in Section \ref{sec:main}. We should note, however, that none of these families allows to reprove \cite[Theorem 1.2]{kassabov2024monotone} (\say{no isolated points}).
\end{rem}

\begin{rem}
	We could replace the use of Corollary \ref{cor:rho_cont} by \cite[Theorem 2]{Ollivier}, which implies that $\rho(G/\!\lla w_\ell\rra,S)\to \rho(G,S)$ in probability, where $G$ is a fixed non-elementary torsion free hyperbolic group and $w_\ell\in S^*$ is a random word of length $\ell$.
\end{rem}



\medbreak
\section{Further remarks}
\medbreak

\newcommand{\cube}{\mathrm{cube}}
\subsection{An explicit example} \label{ssec:Woess} The most natural question is the following:

\begin{question}
	Does there exist a finitely presented group $(G,S)$ with decidable Word Problem and transcendental spectral radius $\rho(G,S)$?
\end{question}

Sarnak asked if the spectral radius is transcendental for $G=\pi_1(\Sigma_2)$ the surface group \cite{Grigorchuk_delaHarpe}, and lots of effort has been done to approximate this value \cite{Tatiana, Gouezel}. Of course, we expect $\rho(G,S)$ to be transcendental even in this case. However, we suggest looking at $G=\Z^d*\Z^d$ (with $d\ge 5$) might be a better try for transcendence. 

We consider the \emph{cubical} set $S_\cube=\{\pm 1\}^d$, which generates a group isomorphic to $\Z^d$. Recall that the \emph{Green series} is defined as
\[ \Gamma_{\Z^d,\cube}(z) = \sum_{n\ge 0} p(n)\cdot z^n = \sum_{n\ge 0} {2n\choose n}^d \left(\frac{z}{2^d}\right)^{2n}. \]
This function is a generalized hypergeometric function, specifically 
\[ \Gamma_{\Z^d,\cube}(z) = \leftindex_{d+1}{F}_{d} \left(\frac12,\ldots,\frac12;1,\ldots,1;z\right).\] 
For $d\ge 5$, Woess proves that $\rho(G,S)=1-\frac1{2\theta}$ where
$\theta = \Gamma_{\Z^d,\cube}(1)$ \cite[p.108]{Woess}. (Note that $\theta$ is the expected number of returns at $\mathbf 0$ of the simple random walk.) We are reduced to prove that the value of a G-function at $z=1$ is transcendental. Even though the transcendence of these values is wide open, we fall into the realm of things mathematicians have thought about, and developed some machinery for. \cite{Fischler_Rivoal,survey_Gfunctions}


\medbreak

\subsection{Computability of $\rho(G,S)$}
Note that we cannot take $X=\{\text{computable reals}\}$ in Theorem \ref{thm:diagonal}. (Actually, by construction/Corollary \ref{cor:rho_is_comp}, the spectral radius will be computable for our examples.) This begs the following question:
\begin{question}
	Does there exist a marked group $(G,S)$ with decidable Word Problem such that $\rho(G,S)$ is not computable? (Equivalently, not upper semi-computable.)
\end{question}
If we denote $p(2m)=\rho^{2m}/f(m)$ (with $f(m)\ge 1$ the sub-exponential correction). This would imply that there exists no computable function $N\colon \N\to\N$ such that
\[ \forall m\ge N(n),\quad  \frac1m\log\!\big(f(m)\big) \le \frac1n, \]
despite $\frac1m\log\!\big(f(m)\big)\to 0$. An approach in the spirit of Theorem \ref*{thm:main} could pass by a positive answer to the following question:
\begin{question}
	Does there exist a finitely generated group $G$ with decidable Word Problem such that, given as input a finite subset $T\subset G$, it is undecidable whether the normal subgroup generated $\lla T\rra_G$ is amenable or not?
\end{question}
Equivalently, let $R(G)$ be the amenable radical of $G$, i.e., the largest amenable normal subgroup. Does there exists a finitely generated group $G$ with decidable Word Problem such that the quotient $G/R(G)$ has undecidable Word Problem?

\subsection{Continuity of $\rho(G,S)$} We have proven that the restriction of $\rho\colon\mathcal M\to(0,1]$ to the subset of $C'(1/6)$ presentations is continuous. It is natural to wonder if this can be extended to other natural subsets of the space of marked groups $\mathcal M$, such as
\begin{align*}
	\mathcal H_{\mathrm{ne}} & \coloneqq \{(G,S) \mid G\text{ is non-elementary hyperbolic}\}, \\
	\mathcal H_{\mathrm{tf}} & \coloneqq \{(G,S) \mid G\text{ is torsion free hyperbolic}\}.
\end{align*}
We prove that $\rho$ is not continuous on $\mathcal H_{\mathrm{ne}}$. We denote $F(X)$ the free group over $X$, and $N_{c,q}(X)$ the free nilpotent group of class $c$ and exponent $q$ over $X$. Note that $N_{c,q}(X)$ is finite as soon as $X$ is finite. Given a prime $p$, we define
\[ G_n = N_{n,p^n}(N_{n,p^n}(s,t)) \rtimes F(s,t), \]
where $F(s,t)$ acts by left-multiplication, using that $F(s,t)\onto N_{n,p^n}(s,t)$. This is an instance of \say{permutational verbal wreath product}. $G_n$ is generated by $s,t$ and the \say{lamp at $e\in N_{n,p^n}(s,t)$}, which we denote $a$. Free groups being residually $p$, we have
\[ \bigl(G_n,\{a,s,t\}\bigr) \overset{n\to\infty}\longto \bigl(F(F(s,t)) \rtimes F(s,t),\{a,s,t\}\bigr) = \bigl(F(a,s,t),\{a,s,t\}\bigr).\]
Meanwhile, since $N_{n,p^n}(N_{n,p^n}(s,t))$ is finite (hence amenable) and $F(F(s,t))$ is non-amenable, we can apply Kesten's theorem twice to get
$$\forall n\ge 0,\qquad \rho(G_n,S)=\rho(F(s,t),\{e,s,t\})>\rho(F(a,s,t),\{a,s,t\}).$$
This leaves the following question open:
\begin{question}
	Is the restriction of the map $\rho\colon\mathcal M\to(0,1]$ to $\mathcal H_{\mathrm{tf}}$ continuous? How about the restriction to hyperbolic groups with trivial finite radical?
\end{question}

\AtNextBibliography{\small}
\printbibliography

\end{document}